\documentclass[a4paper,reqno,12pt,draft]{amsart}
\usepackage{amsmath}
\usepackage{amssymb}
\usepackage{amsfonts}
\usepackage{amsthm}
\usepackage{enumerate}
\usepackage{comment}
\usepackage{cite}
\usepackage[dvips]{graphicx,color}%色付け
\usepackage{delarray}%行列式
\usepackage{ascmac}%枠を付ける
\usepackage{fancybox}%影を付けるなど
\usepackage{lineno}
\usepackage[hidelinks,draft=false]{hyperref}
\theoremstyle{plain}
\newtheorem{theorem}{Theorem}[section]
\newtheorem{corollary}[theorem]{Corollary}

\newtheorem{proposition}[theorem]{Proposition}
\newtheorem{lemma}[theorem]{Lemma}
\newtheorem{remark}[theorem]{Remark}
\newtheorem{example}[theorem]{Example}

\numberwithin{equation}{section}
\numberwithin{equation}{section}

\def\XXint#1#2#3{{\setbox0=\hbox{$#1{#2#3}{\int}$}
\vcenter{\hbox{$#2#3$}}\kern-.5\wd0}}

\begin{document}
\title[Bifurcation diagrams in two dimensions]{Bifurcation diagrams of semilinear elliptic equations for supercritical nonlinearities in two dimensions}
\author{Kenta Kumagai}
\address{Department of Mathematics, Tokyo Institute of Technology}
\thanks{This work was supported by JSPS KAKENHI Grant Number 23KJ0949}
\email{kumagai.k.ah@m.titech.ac.jp}
\date{\today}

\begin{abstract}
We consider the Gelfand problem with general supercritical nonlinearities in the two-dimensional unit ball. In this paper, we prove the non-existence of an unstable solution for any positive small parameter $\lambda$. The result implies that once the bifurcation curve emanates from the starting point, then the curve never approaches $\lambda=0$. As a result, we obtain the existence of a radial singular solution.
In addition, we prove the uniformly boundedness of finite Morse index solutions.
As a result, we prove that the bifurcation curve has infinitely many turning points. We remark that these properties are well-known in $N$ dimensions with $3\le N \le 9$ and less known in two dimensions. Our results clarify that the bifurcation structure is solely determined by the supercriticality of the nonlinearities if $2\le N\le 9$.
\end{abstract}

\maketitle

\raggedbottom

\section{Introduction}
Let $N=2$ and $B_1\subset \mathbb{R}^N$ be the unit ball.
We consider the elliptic equation
\begin{equation}
\label{gelfand}
\left\{
\begin{alignedat}{4}
 -\Delta u&=\lambda f(u)&\hspace{2mm} &\text{in } B_1,\\
u&>0  & &\text{in } B_1, \\
u&=0  & &\text{on } \partial B_1,
\end{alignedat}
\right.
\end{equation}
where $\lambda>0$ is a parameter. Throughout this paper, we
assume that $f$ satisfies the following \eqref{asf1}, \eqref{asf0} and \eqref{asf2}.
\begin{equation}
\label{asf1}
\text{$f\in C^2[0,\infty)$ such that $f(0)\ge 0$, 
$f'\ge 0$ and $f''\ge 0$.} 
\tag{$f_1$}
\end{equation}
\begin{equation}
\label{asf0}
f'(0)>0 \hspace{6mm}\text{if $f(0)=0$}.  
\tag{$f_2$}
\end{equation}
\begin{equation}
\label{asf2}
\mathcal{A}:=\limsup_{u\to\infty}\left(\frac{F(u)\log F(u)}{f(u)}\right)'<\frac{1}{2},
\tag{$f_3$}
\end{equation}
where 
\begin{equation*}
F(u):=\int_{0}^{u}f(s)\,ds.
\end{equation*}
The assumption \eqref{asf2} implies the supercriticality of $f$ in the sense of the Trudinger-Morser imbedding (see \cite{Tru, Mor}). Indeed, we can confirm that $f(u)=e^{u^p}$ satisfies \eqref{asf2} for $p>2$. For the computation of $\mathcal{A}$ for typical nonlinearities, see Example \ref{ugu}. On the other hand, it follows from Lemma \ref{basiclem} and Remark \ref{basicrem} below that if $f$ satisfies \eqref{asf1}, \eqref{asf0} and \eqref{asf2}, then there exist $q>2$, $c>0$ and $C>0$ such that $f(u)\ge C \exp(c^q u^q)$ for any $u$ sufficiently large.

We are interested in the bifurcation structure of \eqref{gelfand}. By the symmetric result of Gidas, Ni and Nirenberg \cite{Gidas}, each solution of \eqref{gelfand} is radially symmetric and decreasing; thus \eqref{gelfand} is reduced to an ODE problem. Moreover, the set of solutions of \eqref{gelfand} is an unbounded analytic curve represented by $\{(\lambda(\alpha), u(r, \alpha)) ; \alpha>0\}$, where $u(r,\alpha)$ is the solution satisfying $\lVert u \rVert_{L^{\infty}(B_1)}=u(0)=\alpha$. In addition, we remark that the value $f(0)$ makes a difference to the bifurcation structure. In fact, when $f(0)>0$, the bifurcation curve has the stable branch emanating from $(0,0)$ and going to some $\lambda=\lambda^{*}$, and turning back at $\lambda^{*}$, where no solution exists for $\lambda>\lambda^{*}$. Here, we say that a solution $u$ of \eqref{gelfand} is stable if the principal eigenvalue of the linearized operator $-\Delta-\lambda f'(u)$ is nonnegative. On the other hand, when $f(0)=0$, no solution exists for $\lambda>\frac{\lambda_1}{f'(0)}$, where $\lambda_1 >0$ is the first eigenvalue of the Dirichlet-Laplacian in the $N$-dimensional unit ball. Moreover, the bifurcation curve does not have the stable branch and the bifurcation curve is emanating from $(\frac{\lambda_1}{f'(0)},0)$. We note that these properties are satisfied for supercritical nonlinearities even when $3\le N\le 9$. 
We refer to \cite{korman,Dup,Mi2014, Mi2015,Guowei,cabrecapella, korman2014} for details in this paragraph.

When $3\le N\le 9$, in addition to the properties stated above, the following properties for the bifurcation structure have been shown for supercritical nonlinearities.
\begin{itemize}
    \item[{(a)}] Equation \eqref{gelfand} has no unstable solution when the parameter $\lambda$ is sufficiently small, i.e., the bifurcation curve emanating from the starting point never approaches to $\lambda=0$ again.   
    \item [{(b)}] There exists a unique radial singular solution $(\lambda_{*}, U_{*})$ such that the bifurcation curve converges to the singular solution.
    \item [{(c)}] The bifurcation curve turns infinitely many times, i.e., $\lambda(\alpha)$ has infinitely many extreme points.
    \item[{(d)}] If the Morse index of solutions to \eqref{gelfand} are uniformly bounded, then the $L^\infty$-norms of the solutions are uniformly bounded.
\end{itemize}
Here, we say that $U_{*}\in C^2_{\mathrm{loc}}(0,1]$ is a radial singular solution of \eqref{gelfand} for $\lambda=\lambda_{*}$ if $U_{*}(r)$ is a regular solution of \eqref{gelfand} in $(0,1]$ for $\lambda=\lambda_{*}$ so that $U_{*}(r)\to\infty$ as $r\to 0$. In addition, we define the Morse index $m(u)$ as the maximal dimension of a subspace $X\subset C^{\infty}_{0}(B_1)$ such that for any $\xi\in X$, it holds
\begin{equation*}
    \int_{B_1} |\nabla \xi|^2-\lambda f'(u) \xi^2\,dx<0.
\end{equation*}
Property (a) is shown by \cite{Lin} in the case $N\ge 3$ for all nonlinearities satisfying the following supercritical condition
\begin{equation}
\label{lincon}
f(t)t>\left(\frac{2N}{N-2}+\varepsilon\right)F(t), \hspace{4mm}t\ge  t_0 \hspace{6mm}\text{for some $\varepsilon>0$ and $t_0>0$.}
\end{equation}
Moreover, the author also proved the existence of a radial singular solution with the assumption \eqref{lincon}. In addition, there are a number of studies trying to prove property (b) with $N\ge 3$ and property (c) with $3\le N\le 9$ for typical supercritical nonlinearities, such as $f(u)=e^u$, $f(u)=(1+u)^p$ with $p>\frac{N+2}{N-2}$, $f(u)=e^{u^p}$ with $p>1$ and $f(u)=e^{e^u}$. We refer to \cite{Marius, KiWei, chend, Mi2014, Mi2015, Mi2020, Mi2023} for \eqref{gelfand} and \cite{Guowei,Merle, Nor,Flore} for related problems. As a result, property (b) is proved in \cite{Mi2023} for a class of nonlinearities $f$ including these typical nonlinearities but satisfying 
\eqref{lincon} with $N\ge 3$. Moreover, property (c) is proved in \cite{Mi2018,Mi2024} for the same class of $f$ with $3\le N\le 9$. On the other hand, property (d) is proved in \cite{Figalli} with $3\le N\le 9$ for all positive nondecreasing and convex nonlinearities satisfying \eqref{lincon} for all bounded convex domains. Moreover, if additionally the analyticity of $f$ is assumed and the domain is a ball, the authors proved in \cite{Figalli} property (c) as a result of property (d). 
Here, we mention that the gradient estimate 
\begin{equation*}
 0<-ru'< (N-2-\hat{\varepsilon})u   \hspace{6mm} \text{with some $\hat{\varepsilon}>0$ depending only on $\varepsilon$}
\end{equation*}
derived from a Pohozaev-type identity in \cite{Niserrin} (or a Pohozaev-type argument for general domains) and the condition \eqref{lincon} plays a crucial role in proving these properties.

In contrast to the case $3\le N\le 9$, until recently, there have been few studies trying to prove these properties for supercritical nonlinearities in two dimensions. Property (a) is treated in \cite{MM} 
for $f(u)=\lambda_1 u + (e^{u^p}-1)$ with $p>2$ (and the authors announced that the result can be applied to more general nonlinearities with similar growth rates). Moreover, explicit radial singular solutions are given by \cite{Gh,Dh} for $f(u)=e^{u^p}u^{1-2p}$ with $p>1$, $u>1$, and $f(u)=e^{e^{u}-2u}$ with $u>1$.

In two dimensions, a gradient estimate corresponding to that in the case $N\ge 3$ does not follow from 
the Pohozaev-type identity. Instead of this, we focus on the key identity \eqref{POHO} obtained in \cite{tang}, and thus we obtain the gradient estimate \eqref{poho}. In addition, we focus on an interaction between \eqref{poho} and \eqref{asf2}. As a result, we show the properties (a), (c), (d), and part of property (b) for general supercritical nonlinearities in two dimensions. 

%The first goal of this paper is to obtain the non-existence of an unstable solution for any small $\lambda$ with a given 
%supercritical nonlinearity in two dimensions. The assertion implies that once the bifurcation curve emanates from the starting point, then the curve never approaches $\lambda=0$ again. When $N\ge 3$, Lin \cite{Lin} proved the non-existence result for all supercritical nonlinearities. The idea of the proof is based on a gradient estimate of the solutions derived from a Pohozaev-type identity in \cite{Niserrin} and the supercriticality of $f$. On the other hand, the gradient estimate does not follow from the Pohozaev-type identity when $N=2$. In this paper, we apply the key identity \eqref{POHO} obtained in \cite{tang} instead of the Pohozaev-type identity, and thus we obtain the gradient estimate \eqref{poho}. In addition, we focus on an interaction between \eqref{poho} and \eqref{asf2}. As a result, we get the following non-existence result in two dimensions. 
\begin{theorem}
\label{mainthms}
We assume that $N=2$ and $f$ satisfies \eqref{asf1}, \eqref{asf0} and \eqref{asf2}. Then, the following properties are satisfied.
\begin{enumerate}
    \item[\rm{(A)}] There exists some $0<\lambda_0 <\infty$ such that \eqref{gelfand} has no unstable solution for each $0<\lambda<\lambda_0$. More precisely,
\begin{itemize}
\item[{\rm{(i)}}] In the case $f(0)>0$, the equation \eqref{gelfand} has only the minimal solution for each $\lambda<\lambda_0$.
\item[{\rm{(ii)}}] In the case $f(0)=0$, the equation \eqref{gelfand} has
no solution for each $\lambda<\lambda_0$.    
\end{itemize}
\item[\rm{(B)}]
The bifurcation curve of \eqref{gelfand} has at least one asymptotic bifurcation point $\lambda_*\in [\lambda_0,\infty)$. More precisely, for any sequence  $\{\alpha_n\}_{n\in \mathbb{N}}$, there exist a subsequence $\{\alpha_{n_{j}}\}_{j\in \mathbb{N}}$ and a radial singular solution $U_{*}\in C^2_{\mathrm{loc}}(0,1]$ for $\lambda=\lambda_*$ such that $\alpha_{n_j}\to \infty$,
$\lambda(\alpha_{n_j})\to\lambda_*$, and $u(r,\alpha_{n_j})\to U_{*}$ in $C^2_{\mathrm{loc}}(0,1]$ as $j\to \infty$.

\item[\rm{(C)}] 
We assume in addition that $f$ can be extended to an analytic function on $\mathbb{R}$. Then,
the bifurcation curve of \eqref{gelfand} has infinitely many turning points.

\item[\rm{(D)}]
It holds $m(u(r,\alpha))\to\infty$ as $\alpha\to\infty$.
\end{enumerate}
\end{theorem}

\begin{remark}
 \rm{Properties (A), (C), and (D) correspond to properties (a), (c), and (d) in two dimensions. On the other hand, the uniqueness of a radial singular solution is necessary to derive an analogue of property (b) in two dimensions from property (B). It is known in \cite{Mi2020, Mi2023} that the Emden-Fowler transformation plays a key role in order to prove the uniqueness result for $N\ge 3$. On the other hand, this transformation does not work well for $N=2$ and thus it is not guaranteed whether the uniqueness result holds in two dimensions.}   
\end{remark}

We should remark that very recently, Fujishima, Ioku, Ruf and Terraneo \cite{FIRT} proved the existence of a radial singular solution and the asymptotic behavior of the singular solution for a wide class of nonlinearities (including some subcritical nonlinearities). We remark that the result in \cite{FIRT} can be applied to the nonlinearities treated in \cite{Gh,Dh} and some typical nonlinearities, such as $f(u)=u^k e^{u^p}$ with $k\in \mathbb{R}$, $p>1$, $f(u)=e^{u^{p_1}+u^{p_2}}$ with $p_1>1$, $p_2<\frac{p_1}{2}$, and $f(u)=e^{e^u}$ with $k\in \mathbb{R}$. However, the result in \cite{FIRT} can not be applied to $f(u)=e^{u^{p}(\log u)^r}$ with $p>1$, $r\in \mathbb{R}$ and $f(u)=e^{e^{u^p}}$ with $p>1$. Theorem \ref{mainthms} (B) can be applied to all of the above nonlinearities which are supercritical (see \eqref{asf2} and Example \ref{ugu}). We also remark that Naimen \cite{Naimen} very recently proved the result for some supercritical nonlinearities including $f(u)=e^{u^p}$ with $p>2$ by focusing on the concentration phenomenon of the solutions (and in addition, the author also proved that the bifurcation curve oscillates infinitely many times around some $\lambda_{*}$). Our method is based on the identity \eqref{POHO} and the condition \eqref{asf2}, and thus yields the result for general analytic supercritical nonlinearities.

The novelty of this paper is to prove that various properties for the bifurcation curve with supercritical nonlinearities which are commonly observed in the case $3\le N\le 9$ are also valid in two dimensions. We remark that the celebrated paper of Joseph and Lundgren \cite{JL} states that the bifurcation structure exhibits completely different phenomena in two dimensions and the other dimensions for $f(u)=e^u$. 
%More precisely, when $f(u)=e^u$, all of the assertions of our results does not hold in two dimensions while it hold in the case $3\le N\le 9$. 
Here, we point out that $f(u)=e^u$ is subcritical if $N=2$ and supercritical if $N\ge 3$. Our results reveal that regardless of the dimension $N$, the supercriticality of the nonlinearities determines the bifurcation structure in the case $2\le N \le 9$.

This paper consists of four sections. In Section \ref{basicsec}, we compute $\mathcal{A}$ for typical nonlinearities and introduce some properties for nonlinearities satisfying \eqref{asf2}. In addition, we give some notations used throughout this paper. In Section \ref{pohosec}, we combine the condition \eqref{asf2} with the gradient estimate \eqref{poho} derived from the identity \eqref{POHO}. As a result, we prove property (A) and property (B). In Section \ref{oscisec}, we get the $H^1$-estimate for finite Morse index solutions and as a consequence, we show property (C) and property (D).  
%for $\alpha>0$, we define $v=v(r,\alpha):=u\left(\sqrt{\lambda(\alpha)}r,\alpha\right)$. Then, $v$ satisfies
%\begin{equation}
%\label{eqv}
%\left\{
%\begin{alignedat}{4}
% &-\Delta v= f(v), \hspace{14mm}0<r\le \sqrt{\lambda(\alpha)},\\
%&v(0)=1, \hspace{2mm} v'(0)=0, \hspace{2mm} v\left(\sqrt{\lambda(\alpha)}\right)=0.
%\end{alignedat}
%\right.
%\end{equation}
%Our aim is to prove $\lambda(\alpha)>c>0$ for some $c$ independent of $\alpha$. 

%Finally,
%Pohozaev

%\begin{theorem}
%\label{mainthm}
%We assume that $f$ satisfies \eqref{asf1}, \eqref{asf2}. Then, the following properties hold. 
%\begin{itemize}
%    \item[(i)] There exists $0<\lambda_0<\lambda^{*}$ such that problem \eqref{gelfand} has only the minimal solution when $\lambda<\lambda_0$. 
%    \item[(ii)] The bifurcation curve of \eqref{gelfand} has at least one asymptotic bifurcation point 
%    $\lambda_*\in [\lambda_0,\lambda^*)$. More precisely, for any sequence $\{\alpha_n\}_{n\in \mathbb{N}}$ satisfying $\alpha_n\to \infty$, there exist a singular solution $U\in C^2_{\mathrm{loc}}(0,1]$ for $\lambda=\lambda_*$ and a subsequence $\{\alpha_{n_j}\}_{j\in \mathbb{N}}$ such that $\lambda(\alpha_{n_j})\to\lambda_*$ and $u(r,\alpha_{n_j})\to U$ in $C^2_{\mathrm{loc}}(0,1]$ as $j\to \infty$.
 %   \item[(iii)] $m(u(r, \alpha))\to \infty$ as $\alpha \to\infty$.
%In particular, the bifurcation curve of \eqref{gelfand} 
%has infinitely many turning points.    
%\end{itemize}
%\end{theorem}

\section{Preliminaries}
In this section, we introduce basic properties for nonlinearities satisfying \eqref{asf2}. We first compute $\mathcal{A}$ for typical nonlinearities.
\begin{example}
\label{ugu}
Let $0<p_1<\infty$, $p_2<p_1$, $k\in \mathbb{R}$ and
$\mathcal{A}$ be that in \eqref{asf2}. Then, we get the following.
\begin{itemize}
    \item[{\rm{(i)}}] $\mathcal{A}= 1/p_1$ if $f(u)=u^k e^{u^{p_1}}$.
    \item[{\rm{(ii)}}] $\mathcal{A}= 1/p_1$ if $f(u)=e^{u^{p_1}}+e^{u^{p_2}}$ or $f(u)=e^{u^{p_1}+u^{p_2}}$.
    \item[{\rm{(iii)}}] $\mathcal{A}=1/p_1$ if $f(u)=e^{u^{p_1}(\log u)^k}$.
    \item[{\rm{(iv)}}] $\mathcal{A}=0$ if $f(u)=e^{e^{u^{p_1}}-ku}$ or $e^{e^{e^{u^{p_1}}}}$.
\end{itemize}
\end{example}
\begin{proof}
We set $f(u)=u^k e^{u^{p_1}}$ and $\hat{p}=(k-2p_1-1)(k-p_1+1)p{_1}^{-2}$. Then,
for all $u$ sufficiently large, we get 
\begin{align*}
F(u)&= C+\int_{1}^{u} s^{k} e^{s^{p_1}}\,ds=C+\frac{u^{k-p_1+1}}{p_1}e^{u^{p_1}}-\frac{k-p_1+1}{p_1}\int_{1}^{u} s^{k-p_1}e^{s^{p_1}}\,ds\\
&= C+\frac{u^{k-p_1+1}}{p_1}e^{u^{p_1}}-\frac{k-p_1+1}{{p_1}^2}u^{k-2p_1+1}e^{u^{p_1}}+\hat{p}\int_{1}^{u}s^{k-2p_1}e^{s^{p_1}}\,ds,
\end{align*}
where $C>0$ is a positive constant depending only on $f$.
Here, we have 
\begin{align*}
    0\le \int_{1}^{u}s^{k-2p_1}e^{s^{p_1}}\,ds&\le C+\frac{u^{k-3p_1 +1}}{p_1}e^{u^{p_1}}-\frac{k-3p_1+1}{p_1}\int_{1}^{u}s^{k-3p_1}e^{s^{p_1}}\,ds\\
    &\le C+Cu^{k-3p_1+1}e^{u^{p_1}},
\end{align*}
where $C>0$ is a positive constant depending only on $f$. Moreover, we see that
\begin{equation*}
\left(\frac{F(u)\log F(u)}{f(u)}\right)'=1+\log F(u) \left(1-\frac{F(u)f'(u)}{f^2(u)}\right).
\end{equation*}
Therefore, we get (i) by a direct calculation. We remark that (ii), (iii) and (iv) can be obtained by using the methods similar to the proof of (i).
\end{proof}
Next, we prove some basic properties and thus show that the condition \eqref{asf2} implies the supercriticality of $f$ in the sense of the Trudinger-Morser imbedding.
\label{basicsec}
\begin{lemma}
\label{basiclem}
We assume that $f$ satisfies $\eqref{asf1}$ and $\eqref{asf0}$. In addition, we suppose that there exist $u_0>0$, $0<p<\frac{1}{2}$ such that $F(u_0)>1$ and  
\begin{equation}
\label{yoiass}
\left(\frac{F(u)\log F(u)}{f(u)}\right)'<p  \hspace{6mm} \text{for all $u\ge u_0$}.
\end{equation}
Then, we have
\begin{equation}
\label{supercritical}
    \frac{F(u)\log F(u)}{f(u)u}\le p+ \frac{C_1}{u} ,\hspace{8mm} F(u)\ge e^{{(c_1 u)}^{1/p}}\hspace{6mm}\text{for all $u\ge u_0$},
\end{equation}

\begin{equation}
\label{weakassum}
\left(\frac{F(u)}{f(u)}\right)'\le 0 \hspace{6mm} \text{for $u\ge u_0$},
\end{equation}
and 
\begin{equation*}
    f(u)\ge C_2 e^{{(c_1 u)}^{1/p}} \hspace{6mm}\text{for all $u\ge u_0$},
\end{equation*}
where $c_1$, $C_1$ and $C_2$ are positive constants depending only on $u_0$, $p$, and $f$.
\end{lemma}
\begin{remark}
\label{basicrem}
\rm{The assumption \eqref{yoiass} is equivalent to \eqref{asf2}. Indeed, it is easy to see that if $f$ satisfies \eqref{asf1}, \eqref{asf0} and \eqref{asf2}, then \eqref{yoiass} holds for some large $u_0 > 0$.}
\end{remark}

\begin{proof}[Proof of Lemma \ref{basiclem}]
Let $u\ge u_0$. Then, by integrating
\eqref{yoiass} on $[u_0,u]$, we have 
\begin{equation*}
    \frac{F(u)\log F(u)}{f(u)}\le pu+ \frac{F(u_0)\log F(u_0)}{f(u_0)}-pu_0\le pu+C_1\hspace{6mm} \text{for all $u\ge u_0$,}
\end{equation*}
where $C_1$ is depending only on $u_0$. It implies
\begin{equation}
\label{yoiyoi}
\frac{f(u)}{F(u)\log F(u)}\ge \frac{1}{pu+C_1} \hspace{6mm} \text{for all $u\ge u_0$}.
\end{equation}
Integrating
\eqref{yoiyoi} on $[u_0,u]$, we see that
\begin{equation*}
    \int_{u_0}^{u} \frac{f(s)}{F(s)\log F(s)}\,ds\ge \int_{u_0}^{u}\frac{1}{ps+C_1}\,ds.
\end{equation*}
Therefore, we obtain 
\begin{equation*}
   \log \log F(u)-\log \log F(u_0)\ge p^{-1}(\log (pu+C_1)-\log (pu_0+C_1)) \hspace{4mm} \text{for all $u\ge u_0$}.
\end{equation*}
It implies that the inequality \eqref{supercritical} follows for some constants $c_1$ depending only on $u_0$, $p$ and $f$.

On the other hand, from \eqref{yoiass}, we have
\begin{equation*}
1+\log F(u) \left(1-\frac{F(u)f'(u)}{f^2(u)}\right)<p<1 \hspace{6mm}\text{for all $u\ge u_0$}.
    \end{equation*}
Hence, it holds
    \begin{equation*}
   \left( \frac{F(u)}{f(u)}\right)'= 1-\frac{F(u)f'(u)}{f^2(u)}<0 \hspace{6mm}\text{for all $u\ge u_0$}.
    \end{equation*}
    Thus, we have \eqref{weakassum}. In particular, since $F/f$ is nonincreasing, it follows from \eqref{supercritical} that
    \begin{equation*}
     e^{(c_1  u)^{1/p}}\le    F(u)\le \frac{F(u_0)}{f(u_0)}f(u) \hspace{6mm}\text{for all $u\ge u_0$}.
    \end{equation*}
Thus, we get the result.
\end{proof}

%\begin{proposition}
%\label{basicprop}
%Let $N=2$ and we assume that $f$ satisfies $\eqref{asf1}$ and $\eqref{asf2}$. Then, the set of solutions of \eqref{gelfand} is described as $\{\left(\lambda(\alpha), u(r, \alpha)\right); \alpha\in(0,\infty)\}$ with $\alpha:=\lVert u\rVert_{L^\infty(B_1)}=u(0)$. In addition, 
%$\mathcal{C}:=\{(\lambda(\alpha), \alpha); \alpha\in (0,\infty)\}$ is an unbounded analytic curve emanating from $(0,0)$. Moreover, there exist $\lambda^*\in (0,\infty)$ and the mininal branch $\mathcal{K}\subset \mathcal{C}$ emanating from $(0,0)$ such that 
%\begin{itemize}
%\item[(i)] $\lambda(\alpha)$ is bijective if $(\lambda(\alpha),\alpha)\in \mathcal{K}$. We define $\alpha(\lambda)$ as the inverse of $\lambda(\alpha)$. Moreover, $\lambda(\alpha)$ is increasing and  $\mathcal{K}=\{(\lambda, \alpha(\lambda)): 0<\lambda\le\lambda^{*}\}$.
%\item[(ii)] $u(r,\alpha(\lambda_1))<u(r,\alpha(\lambda_2))$ for $0\le r<1$ if  $0<\lambda_1<\lambda_2$.
%\item[(iii)] $u(r, \alpha(\lambda))$ is minimal. i.e., if there exists $(\lambda(\alpha'), \alpha')\in \mathcal{C}$ satisfying $\lambda(\alpha')=\lambda$, then, $u(r,\alpha(\lambda))\le u(r,\alpha')$. In particular, $u(r,\alpha(\lambda))$ is stable. i.e., the linearized operator $-\Delta - f'(u)$ is nonnegative in $H^1_{0}(B_1)$.
%\item[(iii)] Problem \eqref{gelfand} has no solution for $\lambda>\lambda^{*}$. Moreover, \eqref{gelfand} has the unique solution $u(r,\alpha(\lambda^{*}))$ for $\lambda=\lambda^{*}$.
%\end{itemize}
%\end{proposition}
\subsection{Some notations}
\label{notesec}
In this subsection, we introduce a useful change of variables and fix some notations. As mentioned in the introduction, the set of solutions of \eqref{gelfand} is represented by $\{ (\lambda(\alpha), u(r, \alpha)); \alpha>0\}$ with $\alpha:=\lVert u \rVert_{L^{\infty}(B_1)}=u(0)$. Moreover, we fix a value $\lambda^{*}>0$ depending only on $f$ such that no solution exists for $\lambda>\lambda^{*}$. For each solution $u(r,\alpha)$, we denote by $v=v(r,\alpha):=u(\frac{r}{\sqrt{\lambda(\alpha)}},\alpha)$.
Then, $v$ satisfies
\begin{equation}
\label{eqv}
\left\{
\begin{alignedat}{4}
&v''+\frac{v'}{r}+f(v)=0, \hspace{14mm}0<r\le \sqrt{\lambda(\alpha)},\\
&v(0)=\alpha, \hspace{2mm} v'(0)=0, \hspace{2mm} v\left(\sqrt{\lambda(\alpha)}\right)=0.
\end{alignedat}
\right.
\end{equation}
Here, we extend $f$ to $\mathbb{R}$ such that $f$ is locally Lipschitz, nonnegative, nondecreasing, and convex. 

In the following, we fix the notations $\lambda(\alpha)$, $u(r,\alpha)$, $\lambda^{*}$, $v(r,\alpha)$ and $f$ in the above sense. 
\section{A non-existence result}
\label{pohosec}
In this section, we consider the solutions
of the equation 
\begin{equation}
\label{eqf(v)}
v'' + \frac{v'}{r} +f(v)=0 \hspace{6mm}\text{for $r>0$.}
\end{equation}
We first introduce a basic property for the equation \eqref{eqf(v)}. 
\begin{lemma}[see \cite{Mi2020}]
\label{minusgradlem}
Let $v\in C^2[0,\infty)$ be a solution of \eqref{eqf(v)}. 
Then, we have  
\begin{equation*}
    v'(r)\le 0\hspace{4mm}\text{for all $r>0$}. 
\end{equation*}
\end{lemma}
\begin{proof}
Argue by contradiction, we assume that there exists $r_0>0$ such that $v'(r_0)>0$. Since
$rv'$ is nonincreasing, we get 
$v'(r)\ge r_0v'(r_0)r^{-1}$ for $r<r_0$ and thus we obtain $v(r)\le v(r_0)-r_0v'(r_0)\log(r_0/r)\to -\infty$ as $r\to 0$. It contradicts $v\in C^2[0,\infty)$.
\end{proof}
Next, we introduce the key identity \eqref{POHO} obtained in \cite{tang}. As mentioned in the introduction, the identity plays an important role in the proof of the main theorems. 
\begin{lemma}
\label{poholem}
We impose \eqref{asf1} and \eqref{asf0}. Let $v\in C^2[0,\infty)$ be a solution of \eqref{eqf(v)}. Then, we have
\begin{equation}
\label{POHO}
G'(r)=\frac{d}{dv}\left(\frac{F(v)}{f(v)}\right)r{v'}^2,
\end{equation}
where
\begin{equation*}
G(r):=\frac{r^2}{2}F(v(r))+\frac{r^2}{4}{v'}^2(r)+\frac{F(v(r))}{f(v(r))}rv'(r).
\end{equation*}
In addition, we assume that there exists some $u_0>0$ such that $F(u_0)>1$, $v(0)>u_0$, and \eqref{weakassum} hold.
Then, we have $G(r)\le 0$ if $v(r)\ge u_0$. In particular, it holds
\begin{equation}
\label{poho}
    -rv'(r)\le \frac{4F(v(r))}{f(v(r))} \hspace{4mm}\text{if $v(r)\ge u_0$}.
\end{equation}
\end{lemma}

\begin{proof}
By a direct calculation, we get
\begin{equation*}
G'(r)=rF(v(r))+\frac{r^2}{2}f(v)v'+\frac{rv'}{2}(v'+rv'')+\frac{d}{dv}\left(\frac{F(v)}{f(v)}\right)r{v'}^2+\frac{F(v)}{f(v)}(rv')'.
\end{equation*}
Since $v$ satisfies \eqref{eqf(v)}, we obtain 
\begin{equation*}
    (rv')'=rv''+v'=-rf(v).
\end{equation*}
Therefore, it follows from \eqref{weakassum} that 
\begin{equation*}
    G'(r)=\frac{d}{dv}\left(\frac{F(v)}{f(v)}\right)r{v'}^2\le 0
\end{equation*}
provided $v(r)>u_0$. Thus, by Lemma \ref{minusgradlem}, we get $G(r)\le G(0)=0$ provided $v(r)>u_0$.
\end{proof}
Lemma \ref{poholem} implies that the identity \eqref{POHO} ensures the gradient estimate \eqref{poho}. In addition, the interaction between the gradient estimate \eqref{poho} and the supercriticality of $f$ works effectively in proving Theorem \ref{mainthms} (A). In fact, as a result of Lemma \ref{poholem}, we obtain the following key proposition.
%The gradient estimate \eqref{poho} plays a key role in order to prove the uniqueness/non-existence result. When $N\ge 3$ and $f$ is supercritical, Lin \cite{Lin} obtained in the gradient estimate $rv'\le \frac{2N}{\rho+1}$ for some $\rho>\frac{N+2}{N-2}$ as a result of the standard Pohozaev-type identity (established in \cite{}). However, when $N=2$, the gradient estimate of $v$ does not follow from this identity. On the other hand, the improved Pohozaev-type identity \eqref{POHO} ensures the gradient estimate \eqref{poho} of $v$ in the case $N=2$. As a result of Lemma \ref{poholem}, we get the following proposition. 
\begin{proposition}
\label{superprop}
We suppose that $f$ satisfies \eqref{asf1} and \eqref{asf0}. In addition, we assume that \eqref{yoiass} and $F(u_0)>1$ hold for some $u_0>0$ and $0<p<\frac{1}{2}$. Then, there exist $\gamma>0$ and $\beta>u_0$ such that we have the following assertion. For every $B\ge \beta$, there exists $r_B>0$ such that if $v\in C^2[0,\infty)$ is a solution of \eqref{eqf(v)} satisfying $v(0)>\gamma B$, then $v(r_B)\ge B$.
\end{proposition}
We explain the sketch of the proof. We take $\gamma$, $\beta$, $B$, $u_0$ and $r_B$ appropriately and argue by contradiction that there exist $0<r_1<r_2<r_B$ such that $v(r_1)=\gamma B$ and $v(r_2)=B$. We point out that when it satisfies $r_1 v'(r_1)\log r_1\le (1-\delta_1)v(r_1)$ for a small $\delta_1$, then by an ODE method, we have a contradiction. On the other hand, when it holds $r_1 v'(r_1)\log r_1> (1-\delta_1)v(r_1)$, we take $r_{*}<r_1$ such that \eqref{pitari} and \eqref{ookii} hold. Then, it follows from the gradient estimate \eqref{poho} and the supercriticality of $f$ that $F(v)\le r^{-2+\delta_0}$ for some $\delta_0 >0 $ if $r_* <r< r_1$. The fact ensures that the difference between $v(r_*)$ and $v(r_1)$ is not very large. As a result, by the ODE argument, we get a contradiction.
\begin{proof}
Let $\delta_1>0$ and $\delta_2>0$ be small constants
%We first point out that thanks to Lemma \ref{basiclem}, it holds \eqref{weakassum}.
so that
\begin{equation*}
 \delta_0:=2-4p(1+\delta_2)(1-\delta_1)^{-1}>0.
\end{equation*}
Next, let $\gamma>2\delta_{1}^{-1}$ and $\beta>u_0$ such that 
\begin{equation}
\label{roughcondi}
    \frac{C_1}{u}<p \delta_2 \hspace{6mm} \text{for $u\ge \beta$},
 \end{equation}
where $C_1$ is that in Lemma \ref{basiclem}.
In addition, we take a small constant $0<r_B<1$ satisfying 
\begin{equation}
\label{rhcondi1}
r^{\delta_0}(\log r)^2+ 2\int_{0}^{r}s^{-1+\delta_0}
|\log s|(1+|\log s|)
\,ds< \beta(1-\delta_1)\delta_1, \hspace{2mm} \text{ $0<r<r_B$}
\end{equation}
and 
\begin{equation}
\label{rhcondi2}
f(\gamma B+1)r_B^2\le \delta_1 \gamma B.
\end{equation}
Now, we prove that the assertion of Proposition \ref{superprop}
holds for $r_B$ above. We recall that thanks to Lemma \ref{basiclem}, it holds \eqref{weakassum} and thus it holds \eqref{poho}. Moreover, by Lemma \ref{minusgradlem}, it holds $v'\le 0$.
Then, we argue by contradiction and assume that there exist $0<r_1<r_2<r_B$ such that $v(r_1)=\gamma B$ and $v(r_2)=B$. 

\vspace{10pt}
\noindent
\textbf{Case 1.} $r_1 v'(r_1)\log r_1\le (1-\delta_1)v(r_1)$ holds.

In this case, we start from this identity
\begin{equation}
\label{daiji}
    (-rv')'=-r\Delta v=rf(v) \hspace{6mm} \text{for all $r>0$}.
\end{equation}
Fixing $r_1<t<r_2$ and integrating the above on $(r_1, t)$, we get
\begin{equation*}
    -t v'(t)=-r_1 v'(r_1)+\int_{r_1}^{t} sf(v)\,ds\le -r_1v'(r_1)+\frac{t^2}{2}f(\gamma B)
\end{equation*}
by Lemma \ref{minusgradlem} and the fact that $f'\ge 0$. This inequality implies
\begin{equation*}
    -v'(t)\le -\frac{r_1}{t}v'(r_1)+\frac{t}{2}f(\gamma B).
\end{equation*}
Integrating the above on $(r_1,r_2)$ and using \eqref{rhcondi2} and \eqref{asf1}, it holds
\begin{equation*}
    v(r_1)-v(r_2)\le -r_1 v'(r_1)\log(r_2/r_1)+\frac{r_2^2}{2}f(\gamma B)\le r_1v'(r_1)\log r_1 + \frac{\delta_1}{2}\gamma B.
\end{equation*}
Therefore, thanks to the assumption of Case $1$, we get
\begin{equation*}
  \frac{1}{2}(\delta_1 \gamma-2)B\le 0.
\end{equation*}
It contradicts $\gamma>2\delta_{1}^{-1}$.

\vspace{10pt}
\noindent
\textbf{Case 2.} $r_1\log r_1 v'(r_1)>(1-\delta_1)v(r_1)$ holds.

In this case, we take $0<r_*<r_1$ such that
\begin{equation}
\label{pitari}
    r_* \log r_{*} v'(r_{*})=(1-\delta_1)v(r_{*}) 
\end{equation}
and 
\begin{equation}
\label{ookii}
    r\log r v'(r)>(1-\delta_1)v(r) \hspace{6mm}\text{for all $r_*<r\le r_1$}.
\end{equation}
We recall that \eqref{weakassum} holds.
Thus, it follows from \eqref{pitari}, \eqref{ookii}, and Lemma \ref{poholem} that
\begin{equation*}
    (1-\delta_1)v(r) \le rv'(r)\log r \le -\frac{4F(v(r))}{f(v(r))}\log r \hspace{6mm} \text{for $r_*\le r\le r_{1}$}
\end{equation*}
On the other hand, thanks to Lemma \ref{basiclem} and the condition \eqref{roughcondi}, we get
\begin{equation*}
    \frac{F(v(r))\log F(v(r))}{f(v)v(r)}\le (1+\delta_2)p \hspace{6mm}\text{for all $r_{*}\le r\le r_1$}.
\end{equation*}
By the two inequalities above, we have
\begin{equation}
\label{estimateF}
F(v(r))=\exp\left(\frac{f(v)v}{F(v)}\cdot \frac{F(v)\log F(v)}{f(v)v}\right)\le r^{-4p(1+\delta_2)(1-\delta_1)^{-1}}=r^{-2+\delta_0}
\end{equation}
for all $r_*\le r\le r_1$. 

We now consider the identity 
\begin{equation*}
(r\log r v')'=r\log r\Delta v + v'= -r\log r f(v)+v'.     
\end{equation*}
Integrating the above on $(r_*, r_1)$, it holds
\begin{align*}
r_1\log r_1 v'(r_1)-r_{*}\log r_{*}  v'(r_{*})=-\int_{r_{*}}^{r_1}s\log s f(v)\,ds+v(r_1)-v(r_{*}).
\end{align*}
Thanks to \eqref{pitari} and 
\eqref{ookii}, we obtain
\begin{align*}
 \delta_1&(v(r_*)-v(r_1))= v(r_{*})-v(r_1)+(1-\delta_1)(v(r_1)-v(r_*))
 \\
 &\le v(r_{*})-v(r_1)+r_1\log r_1 v'(r_1)-r_{*}\log r_{*} v'(r_{*})
 = -\int_{r_{*}}^{r_1}s\log s f(v)\,ds.
\end{align*}
Then, it follows from 
\eqref{ookii} that
\begin{align*}
-\int_{r_{*}}^{r_1}s\log s f(v)\,ds&= -\int_{r_{*}}^{r_1}s\log s (v')^{-1}(f(v)v')\,ds\notag\\
&\le -(1-\delta_{1})^{-1}\int_{r_{*}}^{r_1}s^2(\log s)^2 v^{-1} f(v)v'\,ds\\
&\le -\left(\beta(1-\delta_{1})\right)^{-1}\int_{r_{*}}^{r_1}s^2(\log s)^2 f(v)v'\,ds.
\end{align*}
Moreover, by using \eqref{estimateF} and \eqref{rhcondi1}, we get
\begin{align*}
-&\int_{r_{*}}^{r_1}s^2(\log s)^2 f(v)v'\,ds\notag\\
&= (r_* \log r_*)^2F(v(r_{*}))
-(r_{1}\log r_1)^2F(v(r_1))+2\int_{r_{*}}^{r_1}s\log s(1+\log s)
F(v)\,ds \\
&\le r_{*}^{\delta_0}(\log r_{*})^2+ 2\int_{r_{*}}^{r_1}s^{-1+\delta_0}
|\log s|(1+|\log s|)
\,ds\le \beta(1-\delta_1)\delta_1.
\end{align*}
Therefore, thanks to the three estimates above, we get 
\begin{equation*}
    v(r_{*})\le 1+v(r_1)=1+\gamma B.
\end{equation*}
Thus, fixing $r_{*}<t<r_2$ and integrating \eqref{daiji} on $(r_{*},t)$, we get 
\begin{equation*}
    -tv'(t)=-r_* v'(r_*)+\int_{r_*}^{t} sf(v)\,ds\le -r_*v'(r_*)+\frac{t^2}{2}f(\gamma B+1)
\end{equation*}
by Lemma \ref{minusgradlem} and the fact that $f'\ge 0$. This inequality implies
\begin{equation*}
    -v'(t)\le -\frac{r_*}{t}v'(r_*)+\frac{t}{2}f(\gamma B+1).
\end{equation*}
Integrating the above on $(r_*,r_2)$ and using \eqref{rhcondi2}, it holds
\begin{equation*}
    v(r_*)-v(r_2)\le -r_* v'(r_*)\log(r_2/r_*)+\frac{r_2^2}{2}f(\gamma B+1)\le (1-\delta_1)v(r_{*})+\frac{\delta_1}{2}\gamma B.
\end{equation*}
Since $v$ is nonincreasing, it holds
\begin{equation*}
   \frac{1}{2} \left(\delta_1 \gamma - 2\right)B\le \delta_1 v(r_{*})-\frac{\delta_{1}}{2}\gamma B-v(r_2)\le 0.
\end{equation*}
It contradicts $\gamma>2\delta_{1}^{-1}$.
\end{proof}
Proposition \ref{superprop} tells us that there is a strictly positive lower bound of the first zeros of $v$ for all $v$ such that $v(0)$ is sufficiently large. In addition, by using the fact and the identity \eqref{POHO}, we prove the following a priori estimate.
\begin{proposition}
\label{goodthm}
We assume that $f$ satisfies \eqref{asf1}, \eqref{asf0} and \eqref{asf2}. Then, there exists $\alpha_0>0$ such that if 
$\alpha>\alpha_0$, then, we have
\begin{equation}
\label{aprioriv}
-v'(r)<C/r \hspace{4mm}\text{and}\hspace{4mm} v(r)\le C(1+|\log r|) 
 \hspace{6mm}\text{for all $0<r<\sqrt{\lambda^{*}+1}$},
\end{equation}
where $C>0$ is a constant independent of $\alpha$, $\lambda^{*}$ is that in Subsection \ref{notesec}, and $v$ is extended to $(0,\infty)$ such that $v\in C^2[0,\infty)$ and $v$ satisfies \eqref{eqf(v)} for all $r>0$. In addition, we get
\begin{equation}
    \label{aprioriu}
-u'(r)\le \frac{C}{r} \hspace{4mm}\text{and}\hspace{4mm} u(r)\le C(1+|\log r|) \hspace{4mm}\text{for all $0<r\le 1$},
\end{equation}
where $C>0$ is a constant independent of $\alpha$.
\end{proposition}
\begin{proof}
Since $f$ satisfies \eqref{asf1}, \eqref{asf0} and \eqref{asf2}, we have $F(u_0)>1$ and 
\eqref{yoiass} for some $u_0>1$ and $0<p<\frac{1}{2}$. We take $\gamma>0$ and $\beta>u_0$ which satisfy the assertion of Proposition \ref{superprop} and we denote by $\alpha_0=\gamma\beta$. We remark that thanks to Lemma \ref{basiclem}, it holds \eqref{weakassum}. Let $v=v(r,\alpha)$ satisfying $\alpha>\alpha_0$.
Then, we recall that $v$ is a solution of \eqref{eqv}.
We point out that we can extend $v$ to $[0,\infty)$ such that $v\in C^2[0,\infty)$ and $v$ satisfies \eqref{eqf(v)} for all $r>0$. Since \eqref{weakassum} holds, by Lemma \ref{poholem}, we have
\begin{equation}
\label{nazo}
    -v'(r)\le \frac{4F(v(r))}{rf(v(r))}\le \frac{4F(u_0)}{f(u_0)}r^{-1}=Cr^{-1}
\end{equation}
provided $v(r)\ge u_0$, where $C>0$ is independent of $v$.

We define $r_v:=\sup\{r:v(r)>u_0$\}. Thanks to Lemma \ref{minusgradlem},
we can verify that $v(r_v)=u_0$ and 
$r_v\le \sqrt{\lambda(\alpha)}\le \sqrt{\lambda^{*}}$. Moreover, thanks to Proposition \ref{superprop}, there exists $r_{\beta}$ independent of $v$ such that $v(r_{\beta})\ge\beta\ge u_0$.
Using Lemma \ref{minusgradlem} again, we deduce that $r_{\beta}\le r_v\le \sqrt{\lambda(\alpha)}$ and \eqref{nazo} holds for $0<r\le r_v$.   
On the other hand, since $v$ satisfies \eqref{eqf(v)}, we obtain for $r_v\le r<\sqrt{\lambda^{*}+1}$,
\begin{align*}
    -r v'(r)= -r_v v'(r_v)+\int_{r_v}^{r}f(v)r\,ds&\le C +\frac{1}{2}f(v(r_v))(\lambda^{*}+1)\\
    &\le C +\frac{1}{2}f(u_0)(\lambda^{*}+1)\le C,
\end{align*}
where $C>0$ is a constant independent of $v$. Therefore, we get 
\begin{equation*}
-v'(r)<C/r \hspace{6mm} \text{for $0<r<\sqrt{\lambda^{*}+1}$}.
\end{equation*}
We remark that $v(\sqrt{\lambda^{*}+1})\le 0$. Thus,
by integrating the above inequality on $(r,\sqrt{\lambda^{*}+1})$, we obtain
\begin{equation*}
    v(r)\le C (\log \sqrt{\lambda^{*}+1} - \log r)+v(\sqrt{\lambda^{*}+1})\le C(1+|\log r|),\hspace{2mm}\text{ $0<r<\sqrt{\lambda^{*}+1}$},
\end{equation*}
where $C>0$ is a constant independent of $v$. Thus, we get \eqref{aprioriv}. In addition, we recall that $r_{\beta}\le \sqrt{\lambda(\alpha)}\le \sqrt{\lambda^{*}}$. Hence, we get \eqref{aprioriu} by using \eqref{aprioriv}.
\end{proof}
\begin{proof}[Proof of Theorem \ref{mainthms} \rm{(A)} and \rm{(B)}]
We take any sequence $\alpha_n >0$ such that $\alpha_n\to\infty$ as $n\to \infty$. We denote by $\lambda_{n}:=\lambda(\alpha_{n})$, $u_{n}(r):=u(r,\alpha_n)$, and $v_{n}(r):=v(r,\alpha_n)$. Then, we claim that there exist $\lambda_{*}\in (0,\infty)$, a radial singular solution $U_{*}$ for $\lambda=\lambda_{*}$, and a subsequence $\{n_j\}_{j\in \mathbb{N}}$ such that $n_j\to \infty$, $\lambda_{n_j}\to \lambda_{*}$, and $u_{n_j}\to U_{*}$ as $j\to\infty$. 

Indeed, since $f$ satisfies \eqref{asf1}, \eqref{asf0}, and \eqref{asf2}, we have $F(u_0)>1$ and \eqref{yoiass} for some $u_0>1$ and $0<p<\frac{1}{2}$. We take $\gamma>0$ and $\beta>u_0$ which satisfy the assertion of Proposition \ref{superprop}. Fix $\delta>0$. Then, thanks to Proposition \ref{goodthm}, we get \eqref{aprioriv} for all $n$ sufficiently large. By \eqref{aprioriv} and the elliptic regularity theory (see \cite{gil}), there exists some $0<\delta_0<1$ such that it holds $\lVert v_{n}\rVert_{C^{2+\delta_0}(B_{\sqrt{\lambda^{*}+1}}\setminus B_{\delta})}<C$ for all $n$ sufficiently large, where $C$ is a constant independent of $n$. Since $\delta>0$ is arbitrary, by the Ascoli-Arzel\`a theorem and a diagonal argument, there exist a subsequence $n_j$ and a radial function $V\in C^2_{\mathrm{loc}}(0,\sqrt{\lambda^{*}+1})$ which satisfies 
$-\Delta V = f(V)$ in $(0,\sqrt{\lambda^{*}+1})$ such that
$v_{n_j}\to V$ in $C^2_{\mathrm{loc}}(0,\sqrt{\lambda^{*}+1})$ as $j\to\infty$. By Lemma \ref{minusgradlem}, we get  
$v_{n_j}(\sqrt{\lambda^{*}})\le 0$ and thus we have $V'\le 0$ and $V(\sqrt{\lambda^{*}})\le 0$. In addition, let $B\ge \beta$. Then, thanks to Proposition \ref{superprop}, there exists $r_{B}>0$ such that $v_{n_{j}}(r_{B})>B$ if $\alpha_{{n}_{j}}>\gamma B$. In particular, $\lim_{r\to 0} V(r)\ge V(r_B)\ge B$. Since $B$ is arbitrary, we deduce that $V$ is a radial singular solution.

We take $\lambda_{*}>0$ which satisfies $V(\sqrt{\lambda_{*}})=0$ and we define $U_{*}(r):=V(\sqrt{\lambda_{*}}r)$. Then, it is easy to see that $U_{*}$ is a radial singular solution of \eqref{gelfand} for $\lambda=\lambda_{*}$ and $\lambda_{*}\le \lambda^{*}$. Moreover, since $v_{n_j}\to V$ in $C^2_{\mathrm{loc}}(0,\sqrt{\lambda^{*}+1})$ as $j\to\infty$, we get $\lambda(\alpha_{n_j})\to \lambda_{*}$ and $u(r,\alpha_{n_j})\to U_{*}$ in $C^2_{\mathrm{loc}}(0,1]$.

Therefore, we get Theorem \ref{mainthms} (B) and $\liminf_{\alpha\to\infty}\lambda(\alpha)>0$. Here, we recall that the bifurcation curve is continuous. Hence, if $f(0)=0$, then we deduce that there exists no solution for any $\lambda$ sufficiently small. On the other hand, when $f(0)>0$, then it is well-known \cite{Dup} that there exists some $\alpha_0>0$ such that if $0<\alpha<\alpha_0$, then $\lambda(\alpha)$ is increasing and $u(r,\alpha)$ is the minimal solution for $\lambda=\lambda(\alpha)$. Therefore, we deduce that there exists only the minimal solution for any $\lambda$ sufficiently small.
\end{proof}

\section{An \texorpdfstring{$H^1$}{LG}-estimate for finite Morse index solutions}
\label{oscisec}
In this section, we prove an $H^1$-estimate for finite Morse index solutions. As a result, we show Theorem \ref{mainthms} (C) and (D). We begin by introducing a basic lemma.
\begin{lemma}
\label{lem2}
Let $z\in B_1$ and $0<\rho<1$ such that $B_{2\rho}(z)\subset B_1$. Then, there exists a Lipschitz function $\xi:\mathbb{R}^2\to [0,1]$ such that $\xi=1$ in $A:=B_{\rho}(z)\setminus B_{\frac{\rho}{2}}(z)$, $\xi=0$ in $\mathbb{R}^2\setminus 2A$, and 
\begin{equation*}
\int_{\mathbb{R}^2} |\nabla \xi|^2\, dx \le C,
\end{equation*}
where $2A:=B_{2\rho}(z)\setminus B_{\frac{\rho}{4}}(z)$ and $C>0$ is a constant independent of $z$ and $\rho$.
\end{lemma}

\begin{proof}
Without loss of generality, we assume that $z=0$. We define $\xi$ as 
\begin{equation*}
\xi(x):=
\begin{cases}
1 &\text{if $\frac{\rho}{2}\le |x|\le \rho$},\\
\frac{\log(\rho/4)- \log (|x|)}{\log (\rho/4)-\log(\rho/2)} &\text{if $\frac{\rho}{4}<|x|<\frac{\rho}{2}$},\\
\frac{\log(2\rho)- \log (|x|)}{\log (2\rho)-\log(\rho)} &\text{if $\rho<|x|<2\rho$},\\
0&\text{otherwise}.
\end{cases}
\,
\end{equation*}
Then, we have  $\xi=1$ in $A:=B_{\rho}(z)\setminus B_{\frac{\rho}{2}}(z)$, $\xi=0$ in $\mathbb{R}^2\setminus 2A$. Moreover, we get
\begin{equation*}
\int_{\mathbb{R}^2} |\nabla\xi|^2\,dx=C\int_{B_{\rho/2}\setminus B_{\rho/4}} \frac{1}{|x|^2}\,dx+C\int_{B_{2\rho}\setminus B_{\rho}} \frac{1}{|x|^2}\,dx\le C. 
\end{equation*}
Thus, we get the result.
\end{proof}
Thanks to Lemma \ref{lem2}, we obtain the following lemma by using the same argument as in the proof of \cite[Lemma 2.2]{Figalli}. 
\begin{lemma}
\label{superlem}
Let $N=2$ and $k\in \mathbb{N}$. We assume that $f$ satisfies \eqref{asf1}, \eqref{asf0} and \eqref{asf2}. If $(\lambda(\alpha), u(r,\alpha))$ is a solution of \eqref{gelfand} such that $m(u(r,\alpha))\le k$, then, we have 
\begin{equation*}
\int_{B_{\rho}}\lambda(\alpha)f'(u)\,dx\le C(1+k)
\end{equation*}
for all $0<\rho<\frac{1}{2}$ satisfying $f'(u)>0$ in $B_{2\rho}$, where $C>0$ is independent of $f$, $u$, and $k$.
\end{lemma}
As a result, we get the following $H^1$-estimate.
\begin{proposition}
\label{h1prop}
We assume that $f$ satisfies \eqref{asf1}, \eqref{asf0}, and \eqref{asf2}.
Let $k\in \mathbb{N}$ and $(\lambda(\alpha), u(r,\alpha))$ be a solution of \eqref{gelfand} such that $m(u(r,\alpha))\le k$. Then, the solution $u=u(r,\alpha)$ satisfies
\begin{equation*}
    \lVert\nabla u \rVert_{L^2(B_1)}\le C,
\end{equation*}
where $C>0$ is depending only on $f$.
\end{proposition}
We point out that the identity \eqref{POHO} and the supercritical condition \eqref{asf2} play a key role in proving Proposition \ref{h1prop}. In fact, the idea of the proof is to observe the identity \eqref{identity}. Then, \eqref{POHO} ensures that the left-hand side of the identity is positive and it is controlled by some constant $C>0$ from above. In addition, \eqref{asf2} and Lemma \ref{superlem} tell us that the first integral of the right-hand side is controlled by some constant $C>0$ from above and
the second integral of the right-hand side is controlled by $C\lVert\nabla u\rVert^2_{L^2(B_r)}$ from below. Therefore, we get the $H^1$-estimate. 
\begin{proof}
Since $f$ satisfies \eqref{asf1}, \eqref{asf0}, and \eqref{asf2}, we have $F(u_0)>1$, $f'(u_0)>0$ and 
\eqref{yoiass} for some $u_0>1$ and $0<p<\frac{1}{2}$.
Then, thanks to the elliptic regularity theory, we 
assume without loss of generality that $\alpha\ge \gamma\beta$, where $\beta$ and $\gamma$ are those in Proposition \ref{superprop}. We take $r_{\alpha}>0$ such that $v(r_{\alpha},\alpha)=u_0$. Then, by Proposition \ref{superprop} and Lemma \ref{minusgradlem}, we deduce that there exists $r_{\beta}>0$ such that it holds $r_{\alpha} \ge r_{\beta}$ for all $\alpha$ satisfying $\alpha\ge \gamma\beta$.

Since $u$ satisfies \eqref{gelfand} for $\lambda=\lambda(\alpha)$, we get
\begin{equation*}
    -\left(\frac{f(u)}{F(u)}ru'\right)'=\frac{\lambda(\alpha)r f^2(u )}{F(u)}-\frac{f'(u)F(u)-f^2(u)}{F^2(u)}{u'}^{2}r.
\end{equation*}
We denote by $\hat{r}:=\frac{r_\beta}{2\sqrt{\lambda_*}}$. Let $r\le \hat{r}$. Integrating the above on $(0,r)$, it holds
\begin{equation}
\label{identity}
    -\frac{f(u)}{F(u)}ru'=\int_{0}^{r}\frac{\lambda(\alpha)f^{2}(u) s}{F(u)}\,ds-\int_{0}^{r}\frac{f'(u)F(u) -f^{2}(u)}{F^{2}(u)}{u'}^{2}s\,ds.
\end{equation}
Here, we remark that thanks to Lemma \ref{basiclem}, it holds \eqref{weakassum} . In addition, thanks to Lemma \ref{poholem} and Lemma \ref{minusgradlem}, it satisfies \eqref{poho}. 
We denote by $t:=\sqrt{\lambda(\alpha)} s$. Then, it follows for all $0<r<\hat{r}$ that
\begin{equation*}
    -\frac{f(u(r))}{F(u (r))}r u'(r)=-\frac{f(v (\sqrt {\lambda(\alpha)} r))}{F(v(\sqrt {\lambda(\alpha)} r))} (\sqrt {\lambda(\alpha)} r) v'(\sqrt {\lambda(\alpha)} r)\le 4, 
\end{equation*}
and
\begin{equation*}
    \frac{f^{2}(u(s))}{F(u(s))}=\frac{f^{2}(v(t)) }{F(v(t))}\le f'(v(t))=f'(u(s)) \hspace{4mm}\text{by \eqref{weakassum}}.
\end{equation*}
Moreover, thanks to \eqref{yoiass} and \eqref{supercritical}, we get
\begin{align*}
    &\frac{f'(u(s))F(u(s)) -f^{2}(u (s))}{F^{2}(u(s))}=\frac{f'(v(t) )F(v(t)) -f^{2}(v(t))}{F^{2}(v (t))}\\
    &=-\log F(v)\frac{f^{2}(v)-f'(v)F(v)}{f^{2}(v)}\cdot \left(\frac{f(v)v}{F(v)\log F(v)}\right)^2 \cdot \frac{\log F(v)}{v^2 }\\
    &\ge (1-p) v^{-2}(p+C_1 v^{-1}(t))^{-2}(c_1v)^{1/p}\ge C v^{\frac{1}{p} - 2}\ge C,
\end{align*}
where $C>0$ is a constant depending only on $f$. In addition, we remark that
it follows from Lemma \ref{minusgradlem} that $f'(u)>0$ for all $r<2\hat{r}$.
Hence, by combining the facts stated above with Lemma \ref{superlem}, we get
\begin{equation*}
    \int_{B_{\hat{r}}} |\nabla u|^2\,dx \le C,
\end{equation*}
where $C>0$ is a constant depending only on $f$. 
Thus, by combining the above estimate with \eqref{aprioriu}, we get the result.
\end{proof}

\begin{proof}[Proof of Theorem \ref{mainthms} \rm{(D)}]
Argue by contradiction, we assume that there exist $k\in \mathbb{N}$ and a sequence $\{\alpha_{n}\}_{n\in \mathbb{N}}$ such that $\alpha_n \to \infty$ and 
$m(u(r,\alpha_n))\le k$. Thanks to Theorem \ref{mainthms} (B), there exist $\lambda_{*}\in (0,\infty)$ and a radial singular solution $U_{*}\in C^2_{\mathrm{loc}}(0,1]$ of \eqref{gelfand} for $\lambda=\lambda_{*}$ such that $\lambda(\alpha_n)\to \lambda_{*}$ and $u(r,\alpha_n)\to U_{*}$ in $C^2_{\mathrm{loc}}(0,1]$ as $n\to \infty$ by taking a subsequence if necessary. We denote by $u_{n}(r):=u(r,\alpha_n)$ and $\lambda_n = \lambda(\alpha_n)$. Thanks to Proposition \ref{h1prop} and the Poincar\'e inequality, it holds  $\lVert\nabla u_n \rVert_{H^1(B_1)}<C$. Therefore, we deduce that $U_{*}\in H^{1}_{0}(B_1)$ and $u_n\rightharpoonup U_{*}$ in $H^1_{0}(B_1)$ by taking a subsequence if necessary. 

Then, we use the method in the proof of \cite[Theorem 4.1]{CFRS} and prove that $f(U_{*})\in L^1_{\mathrm{loc}}(B_1)$ and $-\Delta U_{*} = \lambda_{*} f(U_{*})$ in $B_1$ in the weak sense. Let $\phi\in C^{0,1}_{0}(B_1)$. Without loss of generality, we assume that $\phi\ge 0$. Then, we get
\begin{equation*}
\int_{B_1}\nabla U_{*} \cdot \nabla \phi\,dx= \lim_{n\to\infty}\int_{B_1}\nabla u_n \cdot \nabla \phi\,dx=\lim_{n\to\infty}\int_{B_1}\lambda_n f(u_n)\phi\,dx.
\end{equation*}
Let $j\in \mathbb{N}$ and we take a continuous function $\varphi:\mathbb{R}\to[0,1]$ such that $\varphi=0$ on $(-\infty,0]$ and $\varphi=1$ on $[1,\infty)$. Then, we see that
\begin{align*}
\int_{B_1}\lambda_n f(u_n)\varphi(u_n-j)\phi\,dx\le C\int_{\{u_n>j\}}\lambda_n f(u_n)\,dx&\le 
\frac{C}{j}\int_{\{u_n>j\}}\lambda_n f(u_n)u_n\,dx \\
&\le\frac{C}{j}\int_{B_1} |\nabla u_n|^2\,dx\le \frac{C}{j},
\end{align*}
where $C$ is depending only on $f$ and $\phi$. In particular, by Fatou's Lemma, we have
\begin{equation*}
\int_{B_1}\lambda_{*}f(U_{*})\varphi(U_{*}-j)\phi\,dx\le \frac{C}{j}.
\end{equation*} 
On the other hand, by the dominated convergence theorem, we deduce that
\begin{equation*}
\int_{B_1}\lambda_{n}f(u_n)(1-\varphi(u_n-j))\phi\,dx\to \int_{B_1}\lambda_{*}f(U_{*})(1-\varphi(U_{*}-j))\phi\,dx \hspace{4mm}\text{as $n\to\infty$}.
\end{equation*}
Therefore, we have
\begin{align*}
\lim_{n\to\infty}\int_{B_1}\lambda_n &f(u_n)\phi\,dx\le \lim_{n\to\infty} \int_{B_1}\lambda_n f(u_n)(1-\varphi(u_n-j))\phi\,dx+\frac{C}{j}\\
&=\int_{B_1}\lambda_{*}f(U_{*})(1-\varphi(U_{*}-j))\phi\,dx+\frac{C}{j}\le\int_{B_1}\lambda_{*}f(U_{*})\phi\,dx+\frac{C}{j}
\end{align*}
and
\begin{align*}
\lim_{n\to\infty}\int_{B_1}\lambda_n &f(u_n)\phi\,dx\ge \lim_{n\to\infty} \int_{B_1}\lambda_n f(u_n)(1-\varphi(u_n-j))\phi\,dx\\&
=\int_{B_1}\lambda_{*}f(U_{*})(1-\varphi(U_{*}-j))\phi\,dx \ge \int_{B_1}\lambda_{*}f(U_{*})\phi\,dx -\frac{C}{j}.
\end{align*}
By letting $j\to\infty$, we get the claim. In particular, since $U_{*}\in H^1_{0}(B_1)$, we get $f(U_{*})\in L^1_{\mathrm{loc}}(B_1)$.

In addition, by the proof of \cite[Proposition 2.3]{Figalli}, we see that $f'(U_{*})\in L^1_{\mathrm{loc}}(B_1)$ and $m(U_{*})\le k$. Therefore, by \cite[Proposition 1.5.1]{Dup}, $U_{*}$ is locally stable, i.e., there exists $r_0>0$ such that $U_{*}(r_0)>0$ and $U_{*}$ is stable in $B_{r_0}$. We denote by $W:=U_{*}-U_{*}(r_0)$. Then, $W\in H^1_{0}(B_{r_0})$ is a singular stable solution of \begin{equation*}
\left\{
\begin{alignedat}{4}
 -\Delta w&=\mu g(w)&\hspace{2mm} &\text{in } B_{r_0},\\
w&>0  & &\text{in } B_{r_0}, \\
w&=0  & &\text{on } \partial B_{r_0}
\end{alignedat}
\right.
\end{equation*}
for $\mu=\lambda_{*}$ and $g(w):=f(w+U_{*}(r_0))$. By \cite[Theorem 3.1]{Br}, $W\in H^1_{0}(B_1)$ is the singular extremal solution. It contradicts the result of \cite{cabrecapella}.
\end{proof}

\begin{proof}[Proof of Theorem \ref{mainthms} \rm{(C)}]
We remark that if $(\lambda(\alpha), u(r,\alpha))$ is a solution of \eqref{gelfand}, then it follows that $m(u)<\infty$ and $m(u)$ is equal to the number of negative eigenvalues of the linearized operator $-\Delta - \lambda f'(u)$ in $H^1_{0}(B_1)$ (see \cite[Proposition 1.5.1]{Dup}). We also remark that the bifurcation curve of \eqref{gelfand} is an analytic curve, which has no self-intersection and no secondary bifurcation point (since the curve is parameterized by $\alpha$).
Therefore, thanks to Theorem \ref{mainthms} (D), we get Theorem \ref{mainthms} (C) by using the methods similar to the proof of \cite[Theorem 2.2]{Guowei} (see also \cite{BD,Dan}).
\end{proof}

\subsection*{Acknowledgment}
The author would like to thank Professor Michiaki Onodera, for his valuable advice to improve the presentation. He appreciates Professor Norisuke Ioku for information about his recent work and valuable discussions. He is also grateful to Professor Daisuke Naimen for information about his recent work and stimulating discussions.

\bibliographystyle{plain}
\bibliography{bifurcation_2_revised.bib}

% -----------------------------------------------------------------------
\end{document}